\documentclass{article}
\usepackage{amsfonts}
\usepackage{graphicx}

\newtheorem{theorem}{Theorem}

\newtheorem{corollary}[theorem]{Corollary}

\newtheorem{proposition}[theorem]{Proposition}

\newenvironment{proof}[1][Proof]{\noindent\textbf{#1.} }{\ \rule{0.5em}{0.5em}}
\begin{document}

\title{A Jacobian Group Structure on a Hyperbolic Pencil of circles and its
Applications}
\author{Faruk F. Abi-Khuzam \\
Department of Mathematics\\
American University of Beirut\\
Beirut, Lebanon\\
e-mail farukakh@aub.edu.lb}
\maketitle

\begin{abstract}
Using Jacobian Elliptic functions, we introduce a novel parametrization of a
hyperbolic pencil of coaxal circles which reveals a remarkable group
structure on the pencil. The geometric properties of the group elements lead
to a new proof of of the general Poncelet theorems, which in turn leads to a
proof of the so called closure theorem. In particular we prove: if $T$ and $%
D $ are members of the pencil, then an interscribed $n$-gon to $T$ and $D$
exists, if and only if $D$, the inside circle, is an element of order $n$ in
the group.

\textbf{Mathematics Subject Classification. }51M04, 33M04

\textbf{Keywords}. Coaxal, Hyperbolic Pencil, Jacobian Elliptic Functions,
Poncelet Theorems
\end{abstract}

\section{Introduction}

Let $T$ be one circle in a hyperbolic pencil of coaxal circles \cite{RJ} in
the plane, and let $L$ be the limit point of the pencil interior to $T.$%
There will be no loss of generality if we assume, as we do, that $T$ is the
unit circle in the complex plane $%
\mathbb{C}
.$ Fix the real number $k\in (0,1)$, and let $g:%
\mathbb{R}
\rightarrow 
\mathbb{R}
$ be the function defined by 
\[
u=g(\theta )=g(\theta ;k):=\int_{0}^{\theta }\frac{dt}{\sqrt{1-k^{2}\sin
^{2}t}},\left( -\infty <\theta <\infty \right) . 
\]

We shall use this function to introduce a novel "non-classical"
parametrization of the pencil of coaxal circles determined by the circle $T$
and an interior point $L.$This parametrization reveals a remarkable group
structure on the pencil, and makes it possible to identify each of its
oriented circles with a precise homeomorphism of $T.$ The geometric
interpretation of those homeomorphisms leads to simple proofs of the
Poncelet general theorems.

The general Theorem of Poncelet \cite{RJ} may be phrased as follows:

" \textit{Let }$D_{1},D_{2},..,D_{n}$\textit{\ be }$n$\textit{\ circles of
the Coaxal pencil interior to }$T$\textit{, and let }$P$\textit{\ be a point
on }$T.$\textit{\ A line drawn from }$P$\textit{\ tangent to }$D_{1}$\textit{%
\ meets }$T$\textit{\ again in a point }$P_{1}.$ Continuing, having the
point $P_{k}$, \textit{a line is drawn from }$P_{k}$\textit{\ tangent to }$%
D_{k+1}$ and \textit{\ meeting }$T$\textit{\ in }$P_{k+1}$\textit{.Then
after }$n$\textit{\ steps, the segment }$P_{n}P$\textit{\ will be tangent to
one and the same circle of the pencil independently of the starting point }$%
P $\textit{." }

When $D_{1},D_{2},..,D_{n}$ are the same circle $D$, it may happen that the
segment $P_{n}P$ is tangent to $D$. In this case we obtain what is \textit{%
called an interscribed polygon to }$T$\textit{\ and }$D$\textit{. i.e. a
polygon with vertices on the outer circle }$T$, and sides tangent to the
inner circle $D.$

In this article, given the two circles $T$ and $D$, we use the limit point
of the coaxal pencil determined by $T$ and $D$ to determine the constant $k$
and so the function $g$ defined above. We use $g$ to construct a group, and
then it will follow that an interscribed polygon to $T$ and $D$ exists if
and only if $D$ is an element of order $n$ in the constructed group.

\begin{figure}
\centering
\includegraphics[scale=0.2]{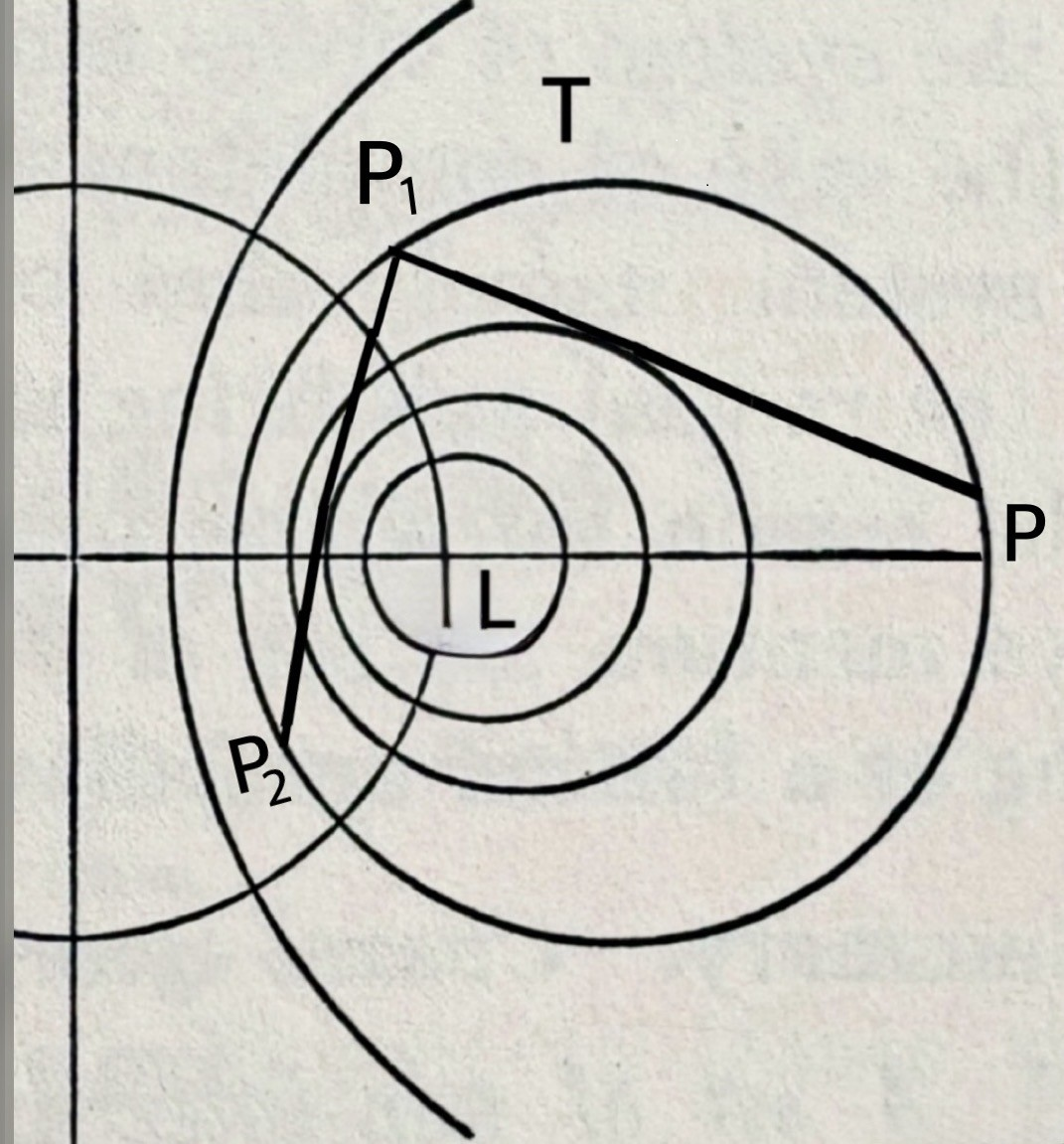}
\end{figure}


In general any two non-intersecting circles $T$ and $D$ where, say, $D$ is
inside but not concentric with $T$, determine a radical axis and a coaxal
pencil of circles symmetrically placed on both sides of the radical axis,
with their line of centers perpendicular to the axis. There will be two
limit points one of which is inside all the circles of the pencil that lie
on the same side of the radical axis. The pencil on one side is determined
by any two of its members. In particular it is determined by one circle and
the limit point inside it. We adopt this point of view here which leads to
the new parametrization alluded to above.

\section{Parametrization of a Coaxal Pencil}

Let $C_{1}$ be a circle inside, but not concentric with, the unit circle $T$%
. The two circles determine a coaxal pencil of circles $\mathcal{C}$.
Suppose, without loss of generality, that $L:=\frac{k^{\prime }-1}{k^{\prime
}+1}$, where $k^{\prime }\in (0,1),$ is the limit point of this pencil
interior to $C_{1}$. Let $k\in (0,1)$ be the index conjugate to $k^{\prime }$%
. i.e $k^{\prime 2}+k^{2}=1.$ Let $g:%
\mathbb{R}
\rightarrow 
\mathbb{R}
$ be the function defined in the introduction with respect to $k$. Then $g$
is a positive, odd, differentiable function, and $g(\theta )\rightarrow
\infty $ as $\theta \rightarrow \infty .$Put $K=g(\pi /2)$, and for $\theta
\in 
\mathbb{R}
,$ put $u=g(\theta ).$ Since $g$ is strictly increasing its inverse function
is well defined and denoted by $\textrm{am},$the "amplitude" function, so
that $\theta =\textrm{am}u$. The three Jacobian Elliptic Functions of $u$
may be defined as follows:

\[
\textrm{cn}u=\cos (\textrm{am}u)=\cos \theta ,\textrm{sn}u=\sin (\textrm{%
am}u)=\sin \theta , 
\]%
\[
\textrm{dn}u=\textrm{dn}(\textrm{am}u)=\sqrt{1-k^{2}\sin ^{2}\theta }. 
\]%
These definitions along with that of $K$ imply that $\textrm{cn}$ is an
even function with $\textrm{cn}K=0$, $\textrm{sn}$ is an odd function with 
$\textrm{sn}K=1,$ and $\textrm{dn}$ is an even function. Furthermore, $%
\textrm{cn}^{2}+\textrm{sn}^{2}=1$, and we have for each integer $n,$ 
\begin{equation}
g(\theta +n\pi )=g(\theta ,k)+2nK,
\end{equation}%
\begin{equation}
\textrm{cn}(u+2nK)=(-1)^{n}\textrm{cn}u;\textrm{sn}(u+2K)=(-1)^{n}%
\textrm{sn}u.
\end{equation}%
In the sequel we shall be using the addition laws for these functions \cite%
{WitWa}.

\begin{equation}
\textrm{cn}(u+v)=\frac{\textrm{cn}u\textrm{cn}v-\textrm{sn}u\textrm{sn}v%
\textrm{dn}u\textrm{dn}v}{1-k^{2}\textrm{sn}^{2}u\textrm{sn}^{2}v}
\end{equation}%
\begin{equation}
\textrm{sn}(u+v)=\frac{\textrm{sn}u\textrm{cn}v\textrm{dn}v+\textrm{sn}v%
\textrm{cn}u\textrm{dn}u}{1-k^{2}\textrm{sn}^{2}u\textrm{sn}^{2}v}
\end{equation}

We first show how to use these Elliptic functions to identify, by a real
number $a$, the individual oriented circles in the coaxal pencil $\mathcal{C}
$ determined by $T$ and $C_{1}$. Since a geometric circle $C$ in the pencil
may be positively or negatively oriented, the number $a$ identifying it
shall be positive if $C$ is positively oriented, and negative if $C$ is
negatively oriented. Also, no orientation shall be attached to the circle $T$%
.

\begin{theorem}
If $C\in \mathcal{C}$ is an oriented circle inside $T$, there is a unique
number $a\in (-K,K]$ such that the center of $C$ is the point $(\frac{%
\textrm{dn}a-1}{\textrm{dn}a+1},0)$ and its radius is $\frac{2\textrm{cn}a%
}{1+\textrm{dn}a}.$ This number $a$ is positive if $C$ is positively
oriented, and negative if $C$ is negatively oriented. Furthermore $C_{0}$ is 
$T$, and $C_{K}$ is the point circle $L$.
\end{theorem}

\begin{proof}
\textit{Denote by }$C_{a}$\textit{\ the circle of center }$(\frac{\textrm{dn%
}a-1}{\textrm{dn}a+1},0)$\textit{\ and radius }$\frac{2\textrm{cn}a}{1+%
\textrm{dn}a}.$\textit{Since the limit point of the pencil }$\mathcal{C}$%
\textit{\ is }$L$\textit{, the radical axis of the system is the line }$x=%
\frac{L^{2}+1}{2L}$\textit{. In order to show that the family of circles }$%
\{C_{a}:a\in (-K,K]\}$\textit{\ is the coaxal pencil determined by }$T$%
\textit{\ and }$L$\textit{\ it suffices to prove that }%
\begin{equation}
\textrm{P}(x,L)=\textrm{P}(x,T)=\textrm{P}(x,C_{a}),
\end{equation}%
\textit{where P}$(x,C)$\textit{\ is the power of point }$x$\textit{\ with
respect to circle }$C$\textit{. i.e. If }$C$\textit{\ has center }$M$\textit{%
\ and radius }$R$\textit{\ then P}$(x,C)=\overline{xM}^{2}-R^{2}.$\textit{\
The first equality is easily verified since }P$(x,L)=\left( \frac{L^{2}-1}{2L%
}\right) ^{2}=$P$(x,T).$\textit{\ To prove the remaining equality we need to
prove that }%
\[
\left( \frac{\textrm{dn}a-1}{\textrm{dn}a+1}-\frac{L^{2}+1}{2L}\right)
^{2}-\left( \frac{2\textrm{cn}a}{\textrm{dn}a+1}\right) ^{2}=\left( \frac{%
L^{2}-1}{2L}\right) ^{2}. 
\]%
\textit{This reduces to }%
\[
(\textrm{dn}a-1)^{2}-\left( 2\textrm{cn}a\right) ^{2}-\frac{L^{2}+1}{L}(%
\textrm{dn}a-1)(\textrm{dn}a+1)+(\textrm{dn}a+1)^{2}=0, 
\]%
\textit{and a straightforward computation gives this last equality. It
follows that the family }$\{C_{a}:a\in (-K,K]\}$\textit{\ is a coaxal pencil
containing }$T$\textit{\ and }$L$\textit{. Hence it is the same coaxal
pencil determined by }$T$\textit{\ and }$C_{1}.$\textit{\ In particular }$%
C_{1}=C_{a}$\textit{\ for some }$a\in (-K,K]\mathit{\cdot }$
\end{proof}

\begin{theorem}
If $a,b\in (-K,K],$let $c\in (-K,K]$ be the unique number such that $a+b$ is
congruent to $c$ mod $2K$. Define $C_{a}\circ C_{b}=C_{c}$. Then the set $%
\{C_{a}:a\in (-K,K]\}$ is a commutative group under this operation, its
identity element is $C_{0}$, the unit circle $T$, and the inverse of $C_{a}$
is $C_{-a}.$
\end{theorem}

\begin{proof}
\ \textit{By \ Theorem 1, every }$c\in (-K,K]$\textit{\ determines a unique
oriented circle of the pencil }$\mathcal{C}$\textit{. So }$\{C_{a}:a\in
(-K,K]\}$\textit{\ inherits the group structure of }$(-K,K]$\textit{\ under
addition mod }$2K.$
\end{proof}

We note that an element $C_{a}$ of order $n$ in this group exists if and
only if $na$ is congruent to $0$ modulo $2K$. If $n=2,$then $C_{K}$ is the
only element of order $2$. If $a\in (0,K)$, and $n\geq 3$, is the smallest
positive integer such that $na$ is congruent to $0$ modulo $2K$, we should
have $na=2hK$ for some positive integer $h$ such that $(h,n)=1$, and $2h<n,$%
\cite{Ja}.\ \ \ \ \ \ \ \ \ \ \ \ \ \ \ \ \ \ \ \ \ \ \ \ \ \ \ \ \ \ \ \ \
\ \ \ \ \ \ \ \ \ \ \ \ \ \ \ \ \ \ \ \ \ \ \ \ \ \ \ \ \ \ \ \ \ \ \ \ \ \
\ \ \ \ \ \ \ \ \ \ \ \ \ \ \ \ \ \ \ \ 

\section{The group $A_{k}$}

In this section we construct a group of homeomorphisms on $T$ using the
function $g$ of the introduction. We retain all the notations used in the
introduction and in section 2.

Fix $\alpha \in 
\mathbb{R}
$, and for $z\in T$, write $z=e^{2i\theta }$, for some $\theta \in 
\mathbb{R}
$. \ If $u$ and $a$ are the numbers defined by 
\begin{equation}
u=g(\theta ,k),a=g(\alpha ,k),
\end{equation}%
let $\theta ^{\prime }$ be the unique number ( in $%
\mathbb{R}
$) satisfying%
\begin{equation}
u+a=g(\theta',k),\textrm{ or }\theta'=\textrm{am}(u+a).
\end{equation}%
Define the mapping $\psi _{\alpha }:T\rightarrow T$ \ by 
\begin{equation}
\psi _{\alpha }(z)=\psi _{\alpha }(e^{2i\theta })=e^{2i\theta ^{\prime }}.
\end{equation}%
For $z\in T,$ the representation $z=e^{2i\theta }$, is not unique, so we
have to show that, $e^{2i(\theta +\pi )^{\prime }}=e^{2i\theta ^{\prime }}$,
and this follows immediately from 
\[
\psi _{\alpha }(e^{2i(\theta +n\pi )})=e^{2i\textrm{am}(g(\theta
)+2nK)}=\left( \textrm{cn}(u+a+2nK)+i\textrm{sn}(u+a+2nK)\right) ^{2} 
\]%
\begin{equation}
=\left( -\textrm{cn}(u+a)-i\textrm{sn}(u+a)\right) ^{2}=\left( -\cos
\theta ^{\prime }-i\sin \theta ^{\prime }\right) ^{2}=e^{2i\theta ^{\prime
}}\cdot
\end{equation}%
so that $\psi _{a}$ is well-defined on $T.$

\begin{theorem}
For a fixed $k\in (0,1)$, the set $A_{k}=\{\psi _{\alpha }:\alpha \in (-%
\frac{1}{2}\pi ,\frac{1}{2}\pi ]\}$ where $\psi _{\alpha }$ is defined by $%
(7)-(8)$, is a commutative group under composition. Its identity element is $%
\psi _{0}$, and the inverse of an element $\psi _{\alpha }$ is $\psi
_{-\alpha }$ . In particular, $\psi _{\pi /2}$ is its own inverse, and it is
the only non-degenerate involution of $T$ in the group $A_{k}$.
\end{theorem}

\begin{proof}
\textit{If }$\alpha ,\beta \in (-\frac{1}{2}\pi ,\frac{1}{2}\pi ]$\textit{\
and we put }%
\[
a=g(\alpha ,k),b=g(\beta ,k), 
\]%
\textit{then, }$-K<a,b\leq K$\textit{, and there is a unique number }$c\in
(-K,K]$\textit{, congruent to }$a+b$\textit{\ modulo }$2K$\textit{. Let }$%
\gamma $\textit{\ be the unique number in }$(-\frac{1}{2}\pi ,\frac{1}{2}\pi
]$\textit{\ satisfying }$c=g(\gamma ,k)$\textit{. Then }$\psi _{\beta }\circ
\psi _{\alpha }=\psi _{\gamma }$\textit{. To see this, take }$z=e^{2i\theta
}\in T$, \textit{and let}%
\[
\psi _{\alpha }(z)=e^{2i\theta ^{\prime }},\psi _{\beta }(e^{2i\theta
^{\prime }})=e^{2i\theta ^{^{\prime \prime }}},\psi _{\gamma
}(z)=e^{2i\theta ^{\prime \prime \prime }}. 
\]%
\textit{If }$u=g(\theta ,k)$\textit{, then according to the definitions }%
\[
u+a=g(\theta ^{\prime },k),u+a+b=g(\theta ^{\prime \prime },k),u+c=g(\theta
^{\prime \prime \prime },k). 
\]%
\textit{Since }$c=a+b+2pK$\textit{\ for some integer }$p\in \{-1,0,1\}$%
\textit{, we have} 
\[
e^{i\theta ^{\prime \prime \prime }}=\textrm{cn}(u+c)+i\textrm{sn}(u+c) 
\]%
\[
=\textrm{cn}(u+a+b+2pK)+i\textrm{sn}(u+a+b+2pK) 
\]%
\[
=(-1)^{p}[\textrm{cn}(u+a+b)+i\textrm{sn}(u+a+b)=(-1)^{p}e^{i\theta
^{\prime \prime }}, 
\]%
\textit{so that }$e^{2i\theta ^{\prime \prime \prime }}=e^{2i\theta ^{\prime
\prime }}$\textit{and it follows that} $\psi _{\beta }\circ \psi _{\alpha
}=\psi _{\gamma }$.\textit{If }$\alpha =0$\textit{, then }$a=0$\textit{, and 
}$g(\theta ^{\prime },k)=u+a=u=g(\theta ,k)$\textit{, implying that }$\theta
^{\prime }=\theta $\textit{\ since }$g$\textit{\ is strictly increasing.
Hence }$\psi _{0}$\textit{\ is the identity map on }$T$\textit{. Of course
it is also the identity element in }$A_{k}$\textit{\ as can be easily
verified. If }$\beta =-\alpha $\textit{, then }$b=-a$\textit{\ and }$c=0$%
\textit{, so that }$\psi _{-\alpha }$\textit{\ is the element inverse to }$%
\psi _{\alpha }$\textit{. In particular, if }$\alpha =\beta =\pi /2$\textit{%
, then }$a=b=K$\textit{\ and }$c=0$\textit{, and this shows that }$\psi
_{\pi /2}$\textit{\ is its own inverse. The group is commutative because }$c$%
\textit{\ depends only on the sum }$a+b$\textit{.}
\end{proof}

\textit{\ \ \ \ \ \ \ \ \ \ \ \ \ \ \ \ \ \ \ \ \ \ \ \ \ \ \ \ \ \ \ \ \ \
\ \ \ \ \ \ \ \ \ \ \ \ \ \ \ \ \ \ \ \ \ \ \ \ \ \ \ \ \ \ \ \ \ \ \ \ \ \
\ \ \ \ \ \ \ \ \ \ \ \ \ \ \ \ \ \ \ \ \ \ \ \ \ \ \ \ \ \ \ }

If $k=0,g(\theta ,k)=\theta $, and the mapping $\psi _{\alpha }$ is a
rotation of $T$ \ through an angle $2\alpha $, in the counterclockwise
sense. Thus we may define $A_{0}$ as the group of rotations of $T,$ which we
identify with $(-\pi /2,\pi /2]$ with the group operation addition modulo $%
\pi $. For the limiting case $k=1$, we take $A_{1}$ to be the group defined
in section 5 below$.$

\section{\protect\bigskip Geometric representation of $A_{k}$}

The two groups introduced above are, completely determined by the index $k,$
the function $g,$ and the corresponding Elliptic functions.
Consequently,there is a clear well defined correspondence between the
elements of the two groups. Indeed to the homeomorphism $\psi _{\alpha },$
corresponds the circle $C_{a}$ where $a=g(\alpha ),$ and this correspondence
is easily shown to be a group isomorphism. This means that a circle $C_{a}$
of the coaxal pencil acts like a function defined on $T$. In this section we
present the common geometric interpretation of both $C_{a}$ and $\psi
_{\alpha .}$. We start with a distance formula.

\begin{proposition}
If $z_{1}$ and $z_{2}$ are distinct points on the unit circle $T$, and $%
z_{3} $ is a point on the real axis, then the distance from $z_{3}$ to the
line joining $z_{1}$ and $z_{2}$ is given by 
\begin{equation}
\frac{1}{2}|z_{3}(1+z_{1}z_{2})-(z_{1}+z_{2})|.
\end{equation}
\end{proposition}

\begin{proof}
\textit{If }$z_{1},z_{2}$\textit{, and }$z_{3}$\textit{\ are three points in
the plane with }$z_{1}\neq z_{2}$\textit{, then the distance from }$z_{3}$%
\textit{\ to the line }$L$\textit{\ joining }$z_{1}$\textit{\ to }$z_{2}$%
\textit{\ is given by the absolute value of the ratio}%
\[
\det \left\vert 
\begin{array}{ccc}
x_{3} & y_{3} & 1 \\ 
x_{2} & y_{2} & 1 \\ 
x_{1} & y_{1} & 1%
\end{array}%
\right\vert \div \sqrt{(x_{2}-x_{1})^{2}+(y_{2}-y_{1})^{2}}, 
\]%
\[
\mathit{=}\det \left\vert 
\begin{array}{ccc}
z_{3} & \bar{z}_{3} & 1 \\ 
z_{2} & \bar{z}_{2} & 1 \\ 
z_{1} & \bar{z}_{1} & 1%
\end{array}%
\right\vert \mathit{\div 2|z}_{2}\mathit{-z}_{1}\mathit{|,} 
\]%
\textit{and a simple computation establishes the required distance formula.}
\end{proof}

We are now ready to represent geometrically the elements of $A_{k}$.

\begin{theorem}
Suppose $\alpha \in (0,\frac{1}{2}\pi )$. If $z_{1}\in T$, and $z_{2}=\psi
_{\alpha }(z_{1})$, then the directed segment $z_{1}z_{2}$ is on the right
side of, and tangent to the positively oriented circle $C_{a}.$ Conversely,
if $z_{1},z_{2}\in T$, and the directed segment $z_{1}z_{2}$ is on the right
side of, and tangent to the positively oriented $C_{a}$, then $z_{2}=\psi
_{\alpha }(z_{1})$ for $0<\alpha <\frac{1}{2}\pi $.
\end{theorem}

\begin{proof}
\textit{It suffices to calculate the distance from the center of }$C_{a}$%
\textit{\ to the line segment }$\overline{z_{1}z_{2}}$\textit{\ by the
distance formula above: recall that, if }$z_{1}=e^{2i\theta }$\textit{, and }%
$z_{2}=\psi _{\alpha }(z_{1})=e^{2i\theta ^{\prime }}$\textit{, then }%
\[
a=g(\alpha ,k),u=g(\theta ,k),u+a=g(\theta ^{\prime },k). 
\]%
\textit{It will be convenient to put }$H=1-k^{2}\textrm{sn}^{2}a\textrm{sn}%
^{2}u$\textit{, and note that }$H=\textrm{cn}^{2}u+\textrm{sn}^{2}u%
\textrm{dn}^{2}a.$\textit{We use the addition laws to obtain }%
\[
(\textrm{dn}a-1)\cos (\theta +\theta ^{\prime })-(\textrm{dn}a+1)\cos
(\theta -\theta ^{\prime }) 
\]%
\[
=-2(\textrm{dn}a\sin \theta \sin \theta ^{\prime }+\cos \theta \cos \theta
^{\prime }) 
\]%
\[
=-2(\textrm{dn}a\textrm{sn}u\textrm{sn}(u+a)+\textrm{cn}u\textrm{cn}%
(u+a)) 
\]%
\[
=-\frac{2}{H}(\textrm{sn}^{2}u\textrm{dn}^{2}a\textrm{cn}a+\textrm{cn}%
^{2}u\textrm{cn}a)=-2\textrm{cn}a. 
\]%
\textit{If we denote the center of }$C_{a}$\textit{\ by }$z_{3}$\textit{,
then we have}%
\[
z_{3}(1+z_{1}z_{2})-(z_{1}+z_{2})=z_{3}(1+e^{2i(\theta +\theta ^{\prime
})})-(e^{2i\theta }+e^{2i\theta ^{\prime }}) 
\]%
\[
=[2z_{3}\cos (\theta +\theta ^{\prime })-2\cos (\theta -\theta ^{\prime
})]e^{i(\theta +\theta ^{\prime })} 
\]%
\[
=\frac{2}{\textrm{dn}a+1}\{(\textrm{dn}a-1)\cos (\theta +\theta ^{\prime
})-(\textrm{dn}a+1)\cos (\theta -\theta ^{\prime })\}e^{i(\theta +\theta
^{\prime })} 
\]%
\[
=-\frac{4\textrm{cn}a}{\textrm{dn}a+1}e^{i(\theta +\theta ^{\prime })}, 
\]%
\textit{and this shows that the distance from the center }$z_{3}$\textit{\
of }$C_{a}$\textit{\ to the segment }$\overline{z_{1}z_{2}}$\textit{\ is }$%
\frac{2\textrm{cn}a}{\textrm{dn}a+1}$\textit{\ which is the length of the
radius. Hence }$\overline{z_{1}z_{2}}$\textit{\ is tangent to }$C_{a}$%
\textit{\ as required.}
\end{proof}

\section{The group $A_{1}$}

In this section we introduce the group $A_{1}$ corresponding to a tangent
pencil of coaxal circles. This group, which may be of independent interest,
requires no use of Elliptic functions. Let $C_{1}$ be a circle internally
tangent to $T$ at the point $(-1,0)$. As before, we represent points $z$ on $%
T$ \ by 
\begin{equation}
z=e^{2i\theta },\left( -\frac{1}{2}\pi <\theta \leq \frac{1}{2}\pi \right) .
\end{equation}%
For a given number $\alpha \in (-\frac{1}{2}\pi ,\frac{1}{2}\pi )$, let $%
\theta ^{\prime }$ be the unique number in $(-\frac{1}{2}\pi ,\frac{1}{2}\pi
]$ defined by the equations%
\begin{equation}
\cos \theta ^{\prime }=\frac{\cos \alpha \cos \theta }{1+\sin \alpha \sin
\theta },\sin \theta ^{\prime }=\frac{\sin \alpha +\sin \theta }{1+\sin
\alpha \sin \theta }\cdot
\end{equation}%
It is immediate that $\cos \theta ^{\prime }$ and $\sin \theta ^{\prime },$
are well-defined and $\cos ^{2}\theta ^{\prime }+\sin ^{2}\theta ^{\prime
}=1 $. Also, $\cos \alpha \cos \theta \geq 0$, for $\theta \in (-\frac{1}{2}%
\pi ,\frac{1}{2}\pi ]$, so that $\cos \theta ^{\prime }$ is always
non-negative, and it follows that there exists exactly one value of $\theta
^{\prime }\in (-\frac{1}{2}\pi ,\frac{1}{2}\pi ]$ that satisfies the two
equations in $(16) $.

Corresponding to the number $\alpha $ we define a mapping $\psi _{\alpha
}:T\rightarrow T$ \ by 
\begin{equation}
\psi _{\alpha }(z)=\psi _{\alpha }(e^{2i\theta })=e^{2i\theta ^{\prime
}},\left( -\frac{1}{2}\pi <\theta \leq \frac{1}{2}\pi \right) ,
\end{equation}%
and note, in particular, that $\psi _{0}(z)=z$ for all $z\in T$.

\begin{theorem}
Let $A_{1}=\{\psi _{\alpha }:-\frac{1}{2}\pi <\alpha <\frac{1}{2}\pi \}$,
where $\psi _{\alpha }$ is defined by $(16-17)$. Then $A_{1}$ is a group
under composition. Its identity element is $\psi _{0}$, and the inverse of
an element $\psi _{\alpha }$ is $\psi _{-\alpha }$.
\end{theorem}

\begin{proof}
\textit{It suffices to show that }$A_{1}$\textit{\ is closed under
composition. If }$\alpha ,\beta \in (-\frac{1}{2}\pi ,\frac{1}{2}\pi )$%
\textit{, and }$\gamma $\textit{\ is the unique number in the same interval
defined by }%
\[
\cos \gamma =\frac{\cos \alpha \cos \beta }{1+\sin \alpha \sin \beta },\sin
\gamma =\frac{\sin \alpha +\sin \beta }{1+\sin \alpha \sin \beta }, 
\]%
\textit{then }$\psi _{\alpha }\circ \psi _{\beta }=\psi _{\gamma }$\textit{,
and composition on }$A_{1}$\textit{\ is a closed operation.}
\end{proof}

The Geometric interpretation of $A_{1}$ is exactly the same as that of $%
A_{k}:$ given a number $\alpha \in (0,\frac{1}{2}\pi )$, we associate to the
mapping $\psi _{\alpha }$ , defined in $(17)$, the positively oriented
circle $C_{1,\alpha }$ of center $(\frac{\cos \alpha -1}{\cos \alpha +1},0)$
and radius $\frac{2\cos \alpha }{\cos \alpha +1}\cdot $ With the same $%
\alpha $, we associate to the mapping $\psi _{-\alpha }$ the negatively
oriented circle $C_{1,-\alpha }$.

Suppose $\alpha \in \lbrack 0,\frac{1}{2}\pi )$. For $0<\alpha <\frac{1}{2}%
\pi $, if $z\in T$, and $z^{\prime }=\psi _{\alpha }(z)$, then the directed
segment $\overrightarrow{zz^{\prime }}$ is on the right side of, and tangent
to the positively oriented circle $C_{1,\alpha }.$ Conversely, if $%
z,z^{\prime }\in T$, and the directed segment $\overrightarrow{zz^{\prime }}$
is on the right side of, and tangent to the positively oriented $C_{1,\alpha
}$, the\textit{n }$\psi _{\alpha }\circ \psi _{\alpha }=\psi _{\gamma }$%
\textit{.}

\section{Applications}

$(a)$ \ \textbf{The Poncelet general theorem}.

Consider a circle $C_{1}$ lying entirely inside the unit circle $T$, with
its center on the radius between $(-1,0)$ and $(0,0)$. Then the two circles
determine a pencil of coaxal circles, with limit point , say $L=\frac{%
k^{\prime }-1}{k^{\prime }+1}$, $0<k^{\prime }<1$. Taking $k=\sqrt{%
1-k^{\prime 2}}$, let $A_{k}$ be the group associated to this pencil. Then
each oriented circle in the system is identified with a mapping $\psi
_{\alpha },\alpha \in (-\frac{1}{2}\pi ,\frac{1}{2}\pi )$.

Starting from a point $P\in T$, draw , in succession a tangent to a first
circle $C_{a_{1}}$of the pencil meeting $T$ in $P_{1}$, and from $P_{1}$
draw a tangent to a second circle $C_{a_{2}}$of the pencil, and so forth;
and let $\alpha _{1},\alpha _{2},...$ be the corresponding numbers defined
by $a_{j}=g(\alpha _{j})$. Put 
\[
P_{1}=\psi _{\alpha _{1}}(P),\cdot \cdot \cdot ,P_{j}=\psi _{\alpha
_{j}}(P_{j-1}),j=1,...,n. 
\]

\begin{theorem}
If $a_{j}\in (-\frac{1}{2}\pi ,\frac{1}{2}\pi )$, where $j=1,2,...,n$, then
there exists a unique $\gamma \in (-\frac{1}{2}\pi ,\frac{1}{2}\pi )$, such
that%
\begin{equation}
\psi _{\alpha _{n}}\circ \cdot \cdot \cdot \circ \psi _{\alpha _{2}}\circ
\psi _{\alpha _{1}}=\psi _{\gamma }.
\end{equation}%
In particular, if $P\in T$, and $P_{j}=\psi _{\alpha _{j}}(P_{j-1}),$ then
the segment $PP_{n}$ remains tangent to the circle $C_{a}$, where $%
a=g(\gamma )$ as $P$ moves on the circle $T.$
\end{theorem}

\begin{proof}
\textit{The closure property of the group }$A_{k}$\textit{\ implies that }$%
P_{n}=\psi _{\gamma }(P)$\textit{\ for some element }$\psi _{\gamma }\in
A_{k}$\textit{. The element }$\psi _{\gamma }$\textit{\ depends on the }$%
\{\alpha _{1},\cdot \cdot \cdot ,\alpha _{n}\}$\textit{, but not on }$P$%
\textit{.Thus there is a circle in the same pencil, namely }$C_{a},$\textit{%
\ such that }$PP_{n}$\textit{\ remains tangent to that circle as }$P$\textit{%
\ moves along }$T$\textit{.}

\begin{corollary}
An interscribed $n$-gon to $C_{1}$ and $C_{2}$ exists, if and only if $C_{1}$
is an element of order $n$ in the group $A_{k}.$Furthermore, if such a
polygon exists then each of its diagonals is tangent to a circle of the
pencil.$.$
\end{corollary}
\end{proof}

As mentioned earlier an interscribed polygon of $n$-sides exists if and only
if $C_{1}$ is an element of order $n$ in $A_{k}$. Since $g(\pi )=2K,$ we
regain\textit{\ the Jacobi condition \cite{Ja} }%
\[
\int_{0}^{\alpha }\frac{dt}{\sqrt{1-k^{2}\sin ^{2}t}}=\frac{h}{n}%
\int_{0}^{\pi }\frac{dt}{\sqrt{1-k^{2}\sin ^{2}t}} 
\]%
\textit{where }$h$\textit{\ is a positive integer and }$(h,n)=1$\textit{.}

\section{Extention of $A_{k}$}

The group $A_{k}$ introduced in the previous sections, is a subgroup of a
larger group that we shall denote by $B_{k}$. We set $\overline{A_{k}}=\{%
\overline{\psi _{\alpha }}:\psi _{\alpha }\in A_{k}\}$, and define 
\[
B_{k}=A_{k}\cup \overline{A_{k}}. 
\]

\begin{theorem}
$B_{k}$ is a group under composition, and $A_{k}$ is a normal subgroup of
it. Every element in $\overline{A_{k}}$ is of order $2,$ and $B_{k}/A_{k}$
is isomorphic to $%
\mathbb{Z}
_{2}$.
\end{theorem}

\begin{proof}
\textit{Since the circles }$C_{k,\alpha }$\textit{\ in }$A_{k}$\textit{\ all
have their centers on the negative real radius of }$T$\textit{, it is easy
to see that}%
\[
\psi _{-\alpha }(\bar{z})=\overline{\psi _{\alpha }(z)}:=\bar{\psi}_{\alpha
}(z)\textrm{,} 
\]%
\textit{and so every element }$\bar{\psi}_{\alpha }$\textit{\ is of order }$%
2 $\textit{. Also }%
\[
\psi _{\alpha }\circ \bar{\psi}_{\beta }(z)=\psi _{\alpha }(\overline{\psi
_{\beta }(z)})=\overline{\psi _{-\alpha }(\psi _{\beta }(z))}=\overline{%
(\psi _{-\alpha }\circ \psi _{\beta })}(z) 
\]%
\textit{and }$\bar{\psi}_{\alpha }\circ \bar{\psi}_{\beta }=\psi _{-\alpha
}\circ \psi _{\beta },$\textit{\ so that }$B_{k}$\textit{\ is closed under
composition. Finally,}%
\[
(\bar{\psi}_{\beta })^{-1}\circ \psi _{\alpha }\circ \bar{\psi}_{\beta }=%
\bar{\psi}_{\beta }\circ \psi _{\alpha }\circ \bar{\psi}_{\beta }=\psi
_{-\alpha } 
\]%
\textit{and }$A_{k}$\textit{\ is a normal subgroup of }$B_{k}$\textit{.}
\end{proof}

We end this section by giving a geometric interpretation of the elements of $%
\overline{A_{k}}$.

\begin{theorem}
Let $\bar{\psi}_{\alpha }\in \overline{A_{k}},\alpha \neq 0$, and $%
a=g(\alpha ,k)$. If $z_{1}\in T$, and $z_{2}=\bar{\psi}_{\alpha }(z_{1})$,
then the line $z_{1}z_{2}$ (extended), remains tangent to the circle of
center $\ z_{3}=\frac{\textrm{dn}a+1}{\textrm{dn}a-1},$ and radius $\frac{2%
\textrm{cn}a}{1-\textrm{dn}a}$.
\end{theorem}

\begin{proof}
\textit{If }$z_{1}=e^{2i\theta }$\textit{, then }$z_{2}=e^{-2i\theta
^{\prime }}$\textit{, where }%
\[
u=g(\theta ,k),a=g(\alpha ,k),u+a=g(\theta ^{\prime },k)\textrm{.} 
\]%
\textit{Then the distance between the center }$z_{3}$\textit{\ and the line }%
$z_{1}z_{2}$\textit{\ is equal to }$\frac{1}{2}$\textit{\ the modulus of} 
\[
z_{3}(1+z_{1}z_{2})-(z_{1}+z_{2}) 
\]%
\[
=\frac{2e^{i(\theta -\theta ^{\prime })}}{\textrm{dn}a-1}\{(\textrm{dn}%
a+1)\cos (\theta -\theta ^{\prime })-(\textrm{dn}a-1)\cos (\theta +\theta
^{\prime })\} 
\]%
\[
=\frac{4\textrm{cn}a}{\textrm{dn}a-1}e^{i(\theta -\theta ^{\prime })}. 
\]
\end{proof}

The circle in this last theorem belongs to the part of the pencil determined
by it and $T$, but lying to the left of the radical axis of the system. Thus
we see that the algebraic structure has now been extended to all circles of
the system, except for those circles in the system containing $T$ but lying
on the same side of the radical axis as $T$.

\textbf{Historical Remark}:The problem of\textit{\ interscribed polygons} ,
and less so of the general problem treated above, has a long history,
starting with Poncelet and extending to modern times. It has attracted the
attention of some great mathematicians, such as Jacobi, Cayley, and
Lebesgue, in addition to Steiner. \ In fact, Poncelet, Caley, and Lebesgue 
\cite{Le} considered the more general problem where the two circles are
replaced by two conics, but Jacobi considered only the circle case. Modern
versions of this problem use the machinary of elliptic curves, and modular
curves. Most recently, in celebration of the 200$^{th}$ anniversary of the
Poncelet result, the paper \ \cite{DrMi}\ \ includes an extensive
bibliography on various variations, and generalizations of the problem.

\section{Compliance with ethical standards}

\textbf{Conflict of interest} The author declares that he has no conflict of
interest.

\bigskip

Faruk Abi-Khuzam

Department of Mathematics

American University of Beirut

Beirut, Lebanon

e-mail: farukakh@aub.edu.lb

\end{document}